\documentclass{amsart}

\usepackage{amsmath}
\usepackage{amsthm}
\usepackage{amssymb}
\usepackage{comment}
\usepackage{xcolor}
\usepackage[colorlinks=false,linkbordercolor=green,citebordercolor=red]{hyperref}
\usepackage[noabbrev, capitalize]{cleveref}
\usepackage{enumerate}
\usepackage{mathtools}
\usepackage{marginnote}

\numberwithin{equation}{section}

\theoremstyle{plain}
\newtheorem{theorem}{Theorem}[section]
\newtheorem{lemma}[theorem]{Lemma}

\newtheorem{proposition}[theorem]{Proposition}

\newtheorem{fact}[theorem]{Fact}

\theoremstyle{definition}
\newtheorem{definition}[theorem]{Definition}

\newtheorem*{acknowledgements}{Acknowledgements}

\theoremstyle{remark}
\newtheorem{remark}[theorem]{Remark}

\theoremstyle{plain}
\makeatletter
\newtheorem{rep@theorem}{Theorem}
\newtheorem{rep@lemma}{Lemma}
\newtheorem{rep@proposition}{Proposition}
\newtheorem{rep@corollary}{Corollary}

\newenvironment{reptheorem}[1]{%
  \begingroup
  \def\therep@theorem{\ref{#1}}%
  \begin{rep@theorem}}
  {\end{rep@theorem}\endgroup}

\makeatother

\newcommand{\N}{\mathbf{N}}
\newcommand{\R}{\mathbf{R}}

\DeclarePairedDelimiter{\abs}{\lvert}{\rvert}
\DeclarePairedDelimiter{\norm}{\lVert}{\rVert}

\title[Completeness conditions for spacetimes with low-regularity metrics]
{Completeness conditions for spacetimes with low-regularity metrics}
\author{Keita Takahashi}
\address{%
Department of Mathematics, Institute of Science Tokyo, O-okayama,
Meguro, Tokyo 152-8551, Japan
}

\date{\today}
\email{takahashi.k.f832@m.isct.ac.jp}

\begin{document}
\begin{abstract}
    We extend Beem's three completeness notions --- finite compactness, timelike Cauchy completeness, and Condition A
    --- originally defined for spacetimes, to Lorentzian length spaces and study their relationships. 
    We prove that finite compactness implies timelike Cauchy completeness and that 
    timelike Cauchy completeness implies Condition A for globally hyperbolic Lorentzian length spaces. 
    Furthermore, for globally hyperbolic $C^{1}$-spacetimes, we establish the equivalence of the 
    three conditions assuming the causally non-branching and non-intertwining conditions, 
    which in fact imply the continuity of the causal exponential map. 
    These results can be regarded as a Hopf-Rinow type theorem for low-regularity Lorentzian geometry. 
    The appendix presents examples of $C^{1}$-spacetimes --- where geodesic uniqueness may fail --- 
    in which causal geodesics nevertheless behave well, illustrating the scope of our results.
\end{abstract}

\maketitle

%\tableofcontents

\section{Introduction}

In Riemannian geometry, a fundamental result is that
every connected Riemannian manifold is naturally endowed with a metric space structure.
The Hopf--Rinow theorem asserts that several completeness conditions
are equivalent in this setting.

\begin{fact}[Hopf--Rinow, cf.~{\cite[Chapter 7, Theorem 2.8]{dC92}}]
Let $(M, g)$ be a connected Riemannian manifold.
The following conditions are equivalent:
\begin{enumerate}
\item $M$ is boundedly compact, i.e., every closed metric ball in $M$ is compact;
\item $M$ is metrically complete, i.e., every Cauchy sequence converges in $M$;
\item $M$ is geodesically complete, i.e., for all $p \in M$, any geodesic 
$\gamma \colon I \to M$ starting at $p$ is defined for all $t \in \R$.
\end{enumerate}
\end{fact}

This theorem can be extended to locally compact length spaces
with slight modifications.

\begin{fact}[Hopf, Rinow, Cohn-Vossen, cf.~{\cite[Theorem 2.5.28]{BBI01}}]
Let $X$ be a locally compact length space.
The following conditions are equivalent:
\begin{enumerate}
\item $X$ is boundedly compact;
\item $X$ is metrically complete;
\item Every geodesic $\gamma \colon [0,a) \to X$ can be extended to a continuous path \break
$\tilde{\gamma} \colon [0,a] \to X$.
\end{enumerate}
\end{fact}

In contrast with the Riemannian case, 
the completeness conditions for Lorentzian manifolds are considerably more involved. 
In particular, for Lorentzian manifolds with low-regularity metrics, 
much less is known, and many fundamental questions remain open.

Busemann~\cite{B67} first introduced finite compactness and timelike Cauchy completeness 
for timelike spaces and proved that finite compactness implies timelike Cauchy completeness.
These conditions are Lorentzian analogues of bounded compactness 
and metric completeness for Riemannian manifolds, respectively. 

Although the specific terminology of timelike spaces introduced by Busemann 
may not appear frequently in standard Lorentzian texts, 
his approach has found interest in modern frameworks.  
Indeed, the structures introduced by Busemann are fundamental to 
currently active subjects of research, particularly concerning the study of 
convex subsets and their duality in Lorentzian geometries (cf.~\cite{FS19}).

Building upon Busemann's work, 
Beem~\cite{B76} proved a following Lorentzian version of the Hopf--Rinow theorem for 
$C^2$-spacetimes. Here, a $C^2$-spacetime means a smooth manifold 
endowed with a time-orientable Lorentzian metric of class $C^2$.    

\begin{fact}[Beem~{\cite[Theorem 5]{B76}}]\label{Beem}
Let $(M, g)$ be a globally hyperbolic $C^2$-spacetime. 
The following conditions are equivalent:
\begin{enumerate}
\item $M$ is finitely compact (cf.~Definition \ref{def:finitely});
\item $M$ is timelike Cauchy complete (cf.~Definition \ref{def:timelike});
\item $M$ satisfies Condition A (cf.~Definition \ref{def:conditionA}).
\end{enumerate}
\end{fact}

A recent result by Burtscher and García-Heveling~\cite{BG24} provides 
a new Lorentzian version of the Hopf--Rinow theorem, 
characterizing global hyperbolicity via the null distance introduced by 
Sormani and Vega~\cite{SV16}, though its current formulation 
appears to require at least $C^{2,1}$ regularity. 

Therefore, the question of completeness remains insufficiently 
understood for spacetimes with Lorentzian metrics of low-regularity. 
Extending Hopf--Rinow type theorems to such low-regularity settings 
is of intrinsic mathematical interest. For instance, Buro~\cite[Theorem 5.22]{B23} established 
a Hopf--Rinow type theorem for weak Finsler metrics of low regularity 
by introducing the notion of special geodesics.

In recent years, there has been significant progress on low-regularity 
Lorentzian geometry via a synthetic approach. 
A central framework in this direction is provided by Lorentzian length spaces, 
introduced by Kunzinger and S\"amann~\cite{KS18}. 
These spaces admit a causal structure and a time separation function 
without assuming a manifold structure or even a Lorentzian metric. 
They provide a general framework extending classical Lorentzian geometry, 
allowing one to study causal and metric properties 
in low-regularity settings. 
For example, Beran, Ohanyan, Rott, and Solis~\cite{BORS23} establish 
a Lorentzian splitting theorem in this setting.

It is thus natural to ask how far Beem's Hopf--Rinow-type implications extend to 
spacetimes with low-regularity Lorentzian metrics and, more broadly, to Lorentzian length spaces. 
Furthermore, within this synthetic framework, we may ask to what extent one can verify which classical completeness 
implications persist---thereby assessing how well Lorentzian length spaces provide a suitable framework.

In this article, we extend Beem's completeness conditions~\cite{B76} 
to Lorentzian length spaces and show that two of Beem's implications
continue to hold in this general setting. 

\begin{reptheorem}{thm:LLS}
Let $(X, \ll, \leq, d, \tau)$ be a globally hyperbolic Lorentzian length space.
\begin{enumerate}
\item If $(X, \ll, \leq, d, \tau)$ is finitely compact, then $(X, \ll, \leq, d, \tau)$ is timelike Cauchy complete. 
\item If $(X, \ll, \leq, d, \tau)$ is timelike Cauchy complete, then $(X, \ll, \leq, d, \tau)$ satisfies Condition A.
\end{enumerate}
\end{reptheorem}

We also show that the remaining implication holds for $C^1$-spacetimes 
under appropriate geodesic assumptions. 
Combined with Theorem~\ref{thm:LLS}, this yields the following result.

\begin{reptheorem}{thm:main}
Let $(M, g)$ be a connected, globally hyperbolic, causally non-branching, 
and non-intertwining $C^1$-spacetime. 
Then the following three conditions are equivalent:
\begin{enumerate}
    \item $M$ is finitely compact;
    \item $M$ is timelike Cauchy complete;
    \item $M$ satisfies Condition A.
\end{enumerate}
\end{reptheorem}

The article is organized as follows.  
In Section 2, we review the notion of Lorentzian length spaces
introduced by Kunzinger and S\"{a}mann~\cite{KS18} and collect some results 
to be used in the proof of the main theorems.  
In Section 3, we extend the three completeness conditions to Lorentzian length spaces 
and establish Theorem \ref{thm:LLS}. 
In Section~4, we study geodesics in $C^1$-spacetimes, 
analyze the Lorentzian distance along causal geodesics, and establish Theorem \ref{thm:main}.
In Appendix A, we present examples of spacetimes with Lorentzian metrics of regularity below $C^{1,1}$ 
for which geodesic uniqueness still holds, even though classical uniqueness results 
for geodesics generally fail in this regularity class. 
In particular, we exhibit $C^1$-spacetimes that are not $C^{1,1}$ and that satisfy 
the assumptions of Theorem \ref{thm:main}. These spacetimes lie outside the scope of 
Fact~\ref{Beem} but are covered by Theorem \ref{thm:main}.

\begin{acknowledgements}
The author would like to thank Sumio Yamada and Kotaro Yamada for helpful discussions 
and valuable advice. The author would also like to thank the two anonymous reviewers for their careful reading and constructive comments. 
In particular, the author is grateful for their suggestion that helped to weaken the assumptions of the main theorem. 
This work was supported by JST SPRING, Grant Number JPMJSP2180.
\end{acknowledgements}

\section{Preliminaries}

In this section, we recall some notions and results on Lorentzian (pre-)length spaces,
which will be used in lemmas and propositions leading to the main theorems.
Throughout the paper, we denote by $\N$ 
the set of natural numbers starting from $1$. 

\begin{definition}[{\cite[Definition 2.1]{KS18}}]
A \emph{causal space} $(X,\ll,\leq)$ is a set $X$ together 
with two binary relations $\ll$ and $\leq$
where both are transitive, $\leq$ is reflexive, and all $x,y \in X$ with $x \ll y$ satisfy $x \leq y.$
We write $x < y$ if $x \leq y$ and $x \neq y$. 
\end{definition}

The \emph{chronological future} $I^{+}(x)$, \emph{chronological past} $I^{-}(x)$,
\emph{causal future} $J^{+}(x)$, and \emph{causal past} $J^{-}(x)$ of $x \in X$ are defined by
\begin{align}
I^{+}(x) &\coloneqq \{y \in X ; x \ll y\},\ I^{-}(x) \coloneqq \{y \in X ; y \ll x\},\\
J^{+}(x) &\coloneqq \{y \in X ; x \leq y\},\ J^{-}(x) \coloneqq \{y \in X ; y \leq x\}.
\end{align}

\begin{definition}[{\cite[Definition 2.8]{KS18}}]\label{LLS}
A \emph{Lorentzian pre-length space} $(X, \ll, \leq, \break d, \tau)$ is a causal space $(X, \ll, \leq)$ together with
a distance $d$ on $X$ and a map $\tau \colon X \times X \to [0, \infty]$ which satisfies the following properties:
\begin{enumerate}[(i)]
\item $\tau$ is lower semicontinuous (with respect to the topology induced by $d$),
\item $\tau(x, z) \geq \tau(x, y) + \tau(y, z)$ for all $x, y, z \in X$ with $x \leq y \leq z,$
\item $\tau(x, y) > 0$ if and only if $x \ll y,$
\item $\tau(x, y) = 0$ if $x \not\leq y.$
\end{enumerate}
$\tau$ is called the \emph{time separation function}.
\end{definition}

\begin{remark}
In the following sections,
we always use the metric topology induced by $d$ 
for all Lorentzian (pre-)length spaces.
\end{remark}

\begin{fact}[{\cite[Lemma 2.10]{KS18}}]\label{pushup}
Let $(X, \ll, \leq, d, \tau)$ be a Lorentzian pre-length space and let
$x,y,z \in X$ with $x \leq y \ll z$ or $x \ll y \leq z.$ Then $x \ll z.$
\end{fact}

\begin{definition}[{\cite[Definition 2.18]{KS18}}]
Let $(X, \ll, \leq, d, \tau)$ be a Lorentzian pre-length space and $I \subset \R$ an interval.
A \emph{future-directed causal} (resp.\ \emph{timelike}) curve is a non-constant, 
locally Lipschitz continuous map
$\gamma \colon I \to X$ satisfying $\gamma(t_1) \leq \gamma(t_2) \ (\text{resp.}\ 
\gamma(t_1) \ll \gamma(t_2))$
for all $t_1, t_2 \in I$ with $t_1 < t_2$. 
A \emph{past-directed causal} (resp.\ \emph{timelike}) curve is a non-constant, 
locally Lipschitz continuous map
$\gamma \colon I \to X$ satisfying $\gamma(t_2) \leq \gamma(t_1) \ (\text{resp.}\ 
\gamma(t_2) \ll \gamma(t_1))$
for all $t_1, t_2 \in I$ with $t_1 < t_2$. 
\end{definition}

We introduce the length of causal curves by using the time separation function.

\begin{definition}[{\cite[Definition 2.24]{KS18}}]
Let $(X, \ll, \leq, d, \tau)$ be a Lorentzian pre-length space and $\gamma \colon [a, b] \to X$
a future-directed causal curve. The \emph{$\tau$-length} of $\gamma$ is defined by
\begin{equation}
L_{\tau}(\gamma) \coloneqq \inf \left\{\sum_{i = 0}^{n-1} \tau(\gamma(t_i), \gamma(t_{i+1})) ; n \in \N, a = t_0 < t_1 < \dots < t_n = b \right\}.
\end{equation}
If $\gamma$ is a past-directed causal curve, the $\tau$-length of $\gamma$ is defined by
\begin{equation}
L_{\tau}(\gamma) \coloneqq \inf \left\{\sum_{i = 0}^{n-1} \tau(\gamma(t_{i+1}), \gamma(t_{i})) ; n \in \N, a = t_0 < t_1 < \dots < t_n = b \right\}.
\end{equation}
\end{definition}

\begin{definition}[{\cite[Definition 2.33]{KS18}}]\label{def:maximal}
Let $(X, \ll, \leq, d, \tau)$ be a Lorentzian pre-length space.
A future-directed causal curve $\gamma \colon [a, b] \to X$ is called \emph{maximal} if
$L_\tau(\gamma) = \tau(\gamma(a), \gamma(b))$. 
\end{definition}

To avoid some technical difficulties, we impose the following 
three conditions, which are used in the definition of Lorentzian length spaces. 
We note that the notion of local causal closedness defined in \cite{KS18}, 
which is included in the definition of Lorentzian length spaces, 
is not appropriate for smooth spacetimes that are not strongly causal (cf.~\cite[Example 2.20]{ACS20}), 
whereas weak local causal closedness defined in \cite[Definition 2.19]{ACS20} is suited. 
However, since we assume global hyperbolicity for Lorentzian length spaces
throughout this article, the two notions are equivalent (cf.~\cite[Proposition 2.21]{ACS20}); 
thus, it does not matter which definition we adopt.

\begin{definition}[{\cite[Definition 3.1]{KS18}}]
A Lorentzian pre-length space $(X, \ll, \leq, \break d, \tau)$ is called \emph{causally path connected} if
for all $x, y \in X$ with $x \ll y$, there exists a future-directed timelike curve from $x$ to $y$, and
for all $x < y$, there exists a future-directed causal curve from $x$ to $y.$
\end{definition}

\begin{definition}[{\cite[Definition 2.19]{ACS20}}]
Let $(X, \ll, \leq, d, \tau)$ be a Lorentzian pre-length space and $x \in X.$
A neighborhood $U$ of $x$ is called \emph{weakly causally closed} if
for any sequences $(p_n)_n, (q_n)_n$ in $U$ satisfying
$p_n \le_U q_n$ for all $n \in \N$ and $p_n \to p \in U,\ q_n \to q \in U$,
then $p \leq_U q$. 
A Lorentzian pre-length space $(X, \ll, \leq, d, \tau)$ is called 
\emph{locally weakly causally closed} if
each point in $X$ has a weakly causally closed neighborhood.
\end{definition}

\begin{definition}[{\cite[Definition 3.16]{KS18}}]
A Lorentzian pre-length space $(X, \ll, \break \leq, d, \tau)$ is called \emph{localizable} if
for all $x \in X, $ there exists an open neighborhood $\Omega_x \subset X$ of $x$ which satisfies the following properties:
\begin{enumerate}[(i)]
\item there is a constant $C > 0$ such that the $d$-length of $\gamma$ 
given by 
\begin{equation}\label{eq:d-arclength}
L^{d}(\gamma) \coloneqq \sup \left\{\sum_{i = 0}^{n-1} d(\gamma(t_i), \gamma(t_{i+1})) ;
n \in \N,
a = t_0 < t_1 < \dots < t_n = b \right\}
\end{equation}
satisfies $L^{d}(\gamma) \leq C$
for all causal curves $\gamma$ contained in $\Omega_x;$
\item there exists a continuous map $\omega_x \colon \Omega_x \times \Omega_x \to [0, \infty)$
which satisfies the following properties:
\begin{enumerate}
\item $(\Omega_x, \ll|_{\Omega_x \times \Omega_x}, \leq|_{\Omega_x \times \Omega_x}, d|_{\Omega_x \times \Omega_x}, \omega_x)$
is a Lorentzian pre-length space,
\item $I^{\pm} (y) \cap \Omega_x \neq \emptyset$ for every $y \in \Omega_x$;
\end{enumerate}
\item for all $p, q \in \Omega_x$ with $p < q$,  
there exists a future-directed causal curve $\gamma_{p, q}$ from $p$ to $q$
which satisfies the following properties:
\begin{enumerate}
\item $L_{\tau}(\gamma_{p, q}) \geq L_\tau(\lambda)$
for all future-directed causal curves $\lambda$ from $p$ to $q$ contained in $\Omega_x,$
\item $L_{\tau}(\gamma_{p, q}) = \omega_x(p, q) \leq \tau(p, q).$
\end{enumerate}
\end{enumerate}
If every point $x \in X$ has a neighborhood basis of open sets $\Omega_x$ 
satisfying (i)-(iii), then $(X, \ll, \leq, d, \tau)$ is called \emph{strongly localizable}. 
\end{definition}

We now define the main object: Lorentzian length spaces.

\begin{definition}[{\cite[Definition 3.22]{KS18}}]
A \emph{Lorentzian length space} is a causally path connected,
locally causally closed, and localizable Lorentzian pre-length space satisfying
\begin{equation}
\tau(x, y)=
\begin{cases}
\displaystyle \sup_{\gamma \in \Omega_{x,y}} L_\tau(\gamma), & \Omega_{x,y} \neq \emptyset,\\
0, & \Omega_{x,y} = \emptyset,
\end{cases}
\end{equation}
where $\Omega_{x,y}$ is the set of all future-directed causal curves from $x$ to $y$. 
\end{definition}

A fundamental source of Lorentzian length spaces comes 
from spacetimes with low-regularity Lorentzian metrics. 
We note that any $C^0$-spacetime can be formulated 
as a Lorentzian pre-length space by using 
nearly timelike curves~\cite[Corollary 3.6]{L25}. 
However, spacetimes with regularity strictly less than $C^{0,1}$ 
may fail to be Lorentzian length spaces 
due to the lack of (weakly) local causal closedness~\cite[Example 3.7]{L25}. 

In contrast, for spacetimes with continuous metrics satisfying 
suitable assumptions, the standard construction yields Lorentzian length spaces. 
Fix any complete Riemannian metric $h$ on $M$ and let $d^h$ be 
the associated Riemannian distance. 
Let $T$ be the usual Lorentzian distance induced by $g$ 
(cf.~Definition \ref{def:Lorentzian-distance}).

\begin{fact}[{\cite[Proposition 5.12]{KS18}}]\label{fact:LSS}
    Let $(M,g)$ be a $C^0$-spacetime that is causally plain
    (cf.~{\cite[Definition 1.16]{CG12}}) and strongly causal
    (cf.~{\cite[Definition 14.11]{O83}}).
    Then, $(M, \ll, \leq, d^h, T)$ is a strongly localizable Lorentzian length space. 
    In particular, if $(M, g)$ is a strongly causal $C^{0,1}$-spacetime, then 
    $(M, \ll, \leq, d^h, T)$ is a strongly localizable Lorentzian length space. 
\end{fact}

To avoid pathological examples, we need to impose some causality conditions.
We now introduce the following conditions.

\begin{definition}[{\cite[Definition 2.35]{KS18}}]\label{non-totally}\label{causal}
A Lorentzian length space $(X, \ll, \leq, \break d, \tau)$ is called
\begin{enumerate}
\item \emph{causal} if $x \leq y$ and $y \leq x$ implies $x=y$ for all $x, y \in X$, 
\item \emph{non-totally imprisoning} if for every compact set $K \subset X$, 
there is $C > 0$ such that the $d$-length in \eqref{eq:d-arclength} of 
all causal curves in $K$ is bounded by $C$, 
\item \emph{globally hyperbolic} if $(X, \ll, \leq, d, \tau)$ is non-totally imprisoning 
and the \emph{causal diamond} $J(x, y) \coloneqq J^{+}(x) \cap J^{-}(y)$ is compact for all $x, y \in X.$
\end{enumerate}
\end{definition}

We collect basic consequences of global hyperbolicity that we shall use repeatedly.   

\begin{fact}[{\cite[Theorem 3.26 (ii)]{KS18}}]\label{imply}
Let $(X, \ll, \leq, d, \tau)$ be a Lorentzian length space. 
If $(X, \ll, \leq, d, \tau)$ is non-totally imprisoning, then 
$(X, \ll, \leq, d, \tau)$ is causal. 
\end{fact}

\begin{fact}[{\cite[Proposition 3.13]{ACS20}}]\label{closed}
Let $(X, \ll, \leq, d, \tau)$ be a globally hyperbolic Lorentzian length space. 
Then the sets $J^{\pm}(x)$ are closed for all $x \in X$. 
\end{fact}

\begin{fact}[{\cite[Theorem 3.28]{KS18}}]\label{finite}\label{cont}
Let $(X, \ll, \leq, d, \tau)$ be a globally hyperbolic Lorentzian length space. 
Then $\tau$ is (i) finite, i.e., $\tau(x, y)\neq\infty$ for all $x,y\in X$, and (ii) continuous.
\end{fact}

\begin{fact}[{\cite[Theorem 3.30]{KS18}}]\label{geodesic}
Let $(X, \ll, \leq, d, \tau)$ be a globally hyperbolic Lorentzian length space. 
Then $(X, \ll, \leq, d, \tau)$ is geodesic, i.e., for all $x, y \in X$ with $x < y$, 
there exists a maximal future-directed causal curve $\gamma$ from $x$ to $y$. 
\end{fact}

\section{Completeness conditions and their implications}

In this section, we extend three conditions---finite compactness, 
timelike Cauchy completeness, and Condition A---to Lorentzian pre-length spaces, 
and prove Theorem \ref{thm:LLS}. 
Theorem \ref{thm:LLS} follows directly from 
Theorem \ref{timelikeCauchycomplete} and Theorem \ref{conditionA}. 

The following definition is a Lorentzian analogue of bounded compactness.
Busemann \cite{B67} first introduced it for timelike spaces.

\begin{definition}\label{def:finitely}
A Lorentzian pre-length space $(X, \ll, \leq, d, \tau)$ is called \emph{finitely compact}
if for every $B > 0$, every $p, q \in X, $ and every sequence of points $(x_n)_n$ in $X$ satisfying
$p \ll q \leq x_n\ (\text{resp.}\ x_n \leq q \ll p)$
and $\tau(p, x_n) \leq B\ (\text{resp.}\ \tau(x_n, p) \leq B)$, 
there is an accumulation point of $(x_n)_n$ in $X$.
\end{definition}

The following lemma is a characterization of finite compactness for
globally hyperbolic Lorentzian length spaces. 

\begin{lemma}\label{compact}
Let $(X, \ll, \leq, d, \tau)$ be a globally hyperbolic Lorentzian length space. Then
$(X, \ll, \leq, d, \tau)$ is finitely compact if and only if
for every $B > 0$ and every $p,q \in X$ with $p \ll q, $
the set $K_1 \coloneqq \{x \in X ; p \ll q \leq x, \tau(p, x) \leq B\}$ is compact, and
for every $B > 0$ and every $p,q \in X$ with $q \ll p, $
the set $K_2 \coloneqq \{x \in X ; x \leq q \ll p, \tau(x, p) \leq B\}$ is compact.
\end{lemma}

\begin{proof}
Assume $(X, \ll, \leq, d, \tau)$ is finitely compact. 
Let $p, q \in X$ with $p \ll q$ and let $K_1$ be as above.
We first show that the set $K_1$ is closed.
We can write $K_1 = J^{+}(q) \cap \{x \in X ; \tau(p,x) \leq B\}.$
By Fact \ref{closed}, $J^{+}(q)$ is closed.
From Fact \ref{finite} (ii), $\{x \in X ; \tau(p,x) \leq B\}$ is also closed.
Thus, $K_1$ is closed. 

Let $A$ be any infinite subset of $K_1$. 
Then there is a sequence $(x_n)_n$ in $A$ with pairwise distinct terms. 
The finite compactness of $X$ shows that the sequence $(x_n)_n$ 
has an accumulation point in $K_1$, and this point is also an accumulation point of $A$.
Thus, $K_1$ is limit-point compact (cf.~{\cite[Chapter 3-7]{M75}}).
By {\cite[Theorem 7.4]{M75}}, $K_1$ is compact. 
Similarly, $K_2$ is compact.

Conversely, assume both $K_1$ and $K_2$ are compact.
Let $B > 0,\ p, q \in X, $ and
$(x_n)_n$ a sequence of points in $X$ satisfying $p \ll q \leq x_n$ and $\tau(p, x_n) \leq B.$
Then by the assumption,
it follows that there exists a subsequence of $(x_n)_n$
which converges to $x_\infty \in K_1.$
Thus, $x_\infty$ is an accumulation point of $(x_n)_n.$
Similarly, for the sequence of points $(x_n)_n$
with $x_n \leq q \ll p$ and $\tau(x_n, p) \leq B$,
we can find an accumulation point of $(x_n)_n.$
\end{proof}

The following definition is a Lorentzian analogue of metric completeness.
Busemann \cite{B67} first introduced it for timelike spaces.

\begin{definition}\label{def:timelike}
A Lorentzian pre-length space $(X, \ll, \leq, d, \tau)$ is called \emph{timelike Cauchy complete}
if for any sequence of points $(x_n)_n$ in $X$ and
any sequence of non-negative numbers $(B_n)_n$ which satisfy
$x_n \ll x_{n+1}\ (\text{resp.}\ x_{n+1} \ll x_n)$ for all $n \in \N, $
$\tau(x_n, x_{n+m}) \leq B_n\ (\text{resp.}\ \tau(x_{n+m}, x_n) \leq B_n)$
for all $n, m \in \N$ and $B_n \to 0$ as $n \to \infty$, 
the sequence $(x_n)_n$ converges in $X$.
\end{definition}

Busemann \cite{B67} first proved the following result for $C^2$-spacetimes. 
We extend this result to globally hyperbolic Lorentzian length spaces.

\begin{theorem}\label{timelikeCauchycomplete}
Let $(X, \ll, \leq, d, \tau)$ be a globally hyperbolic Lorentzian length space.
If $(X, \ll, \leq, d, \tau)$ is finitely compact, then $(X, \ll, \leq, d, \tau)$ is timelike Cauchy complete.
\end{theorem}

\begin{proof}
Assume $(X, \ll, \leq, d, \tau)$ is finitely compact. 
Let $(x_n)_n$ be a sequence of points in $X$ satisfying $x_n \ll x_{n+1}$
and $\tau(x_n, x_{n+m}) \leq B_n$ where $B_n \to 0.$
Define $B \coloneqq \sup_{n \in \N} B_n$. Then $B$ is finite since $B_n \to 0.$
By setting $p \coloneqq x_1$ and $q \coloneqq x_2$,  
the finite compactness of $X$ shows that there is 
a point $x_\infty \in X$ such that $x_{n(k)} \to x_\infty.$

In order to prove the original sequence $(x_n)_n$ converges to $x_\infty$, 
we claim the following:
\begin{align}\label{point}
\bigcap_{k \in \N} (J^{+}(x_{n(k)})) \cap J^{-}(x_\infty) = \{x_\infty\}.
\end{align}

In fact, \eqref{point} can be proved as follows. 
First, let $y \in \bigcap_{k \in \N} (J^{+}(x_{n(k)})) \cap J^{-}(x_\infty), $
then we see that $x_{n(k)} \in J^{-}(y)$ for all $k \in \N$ and $y \leq x_\infty.$
By Fact \ref{closed}, we see that $x_\infty \in J^{-}(y)$, which implies $x_\infty \leq y.$
Thus, we obtain $x_\infty \leq y \leq x_\infty.$
From Fact \ref{imply}, we see that $y=x_\infty$. 
Therefore, $\bigcap_{k \in \N} (J^{+}(x_{n(k)})) \cap J^{-}(x_\infty) \subset \{x_\infty\}$. 

For the reverse inclusion, it follows from the reflexivity of $\leq $ that $x_\infty \in J^{-}(x_\infty)$. 
In order to show that $x_\infty \in \bigcap_{k \in \N} J^{+}(x_{n(k)})$, 
suppose that there exists $K \in \N$ such that
$x_\infty \not\in J^{+}(x_{n(K)})$. 
Then, since $X$ is a regular topological space,
there exist open sets $V$ and $W$ satisfying
$J^{+}(x_{n(K)}) \subset V,\ x_\infty \in W$, and $V \cap W = \emptyset$. 
Note that for all $k \geq K$, $J^+(x_{n(k)}) \subset J^+(x_{n(K)})$ holds 
since $x_n \ll x_{n+1}$ and $x_{n(k)} \in J^{+}(x_{n(k)}) \subset V$ 
for all $k \ge K$, so $(x_{n(k)})_k$ cannot converge to $x_\infty \in W$, a contradiction.
This completes the proof of \eqref{point}.

From \eqref{point}, we see that for any open neighborhood $U$ of $x_\infty$, 
there exists $K \in \N$ such that
$J^{+}(x_{n(K)}) \cap J^{-}(x_\infty) \subset U$. 
Indeed, suppose that for all $k \in \N$, 
$F_k \coloneqq (J^{+}(x_{n(k)}) \cap J^{-}(x_\infty)) \setminus U \neq \emptyset$.  
Each $J^{+}(x_{n(k)}) \cap J^{-}(x_\infty)$ is compact by global hyperbolicity, hence 
$F_k$, being a closed subset thereof, is also compact.
Thus, $(F_k)_k$ is a decreasing sequence of closed compact non-empty sets. 
By Cantor's intersection theorem, $\cap_{k \in \N}F_k \neq \emptyset$. 
This contradicts \eqref{point} since $x_\infty \not \in F_k$.  
This implies that the original sequence $(x_n)_n$ converges to $x_\infty$. 
\end{proof}

We now introduce the notion of inextendible causal curves and geodesics 
to define Condition A (Definition \ref{def:conditionA}).

\begin{definition}[{\cite[Definition 3.10]{KS18}}]
Let $(X, \ll, \leq, d, \tau)$ be a Lorentzian pre-length space and
let $\gamma \colon [a, b) \to X\ (-\infty < a < b \leq \infty)$ be a future-directed
causal (resp. timelike) curve.
The curve $\gamma$ is called \emph{future extendible}
if there exists a future-directed causal (resp. timelike) curve
$\tilde{\gamma} \colon [a, b] \to X$ such that $\tilde{\gamma}|_{[a,b)} = \gamma.$
The curve $\gamma$ is called \emph{future inextendible} if it is not future extendible.
Similarly, we define the notion of past (in)extendibility.
\end{definition}

\begin{definition}[{\cite[Definition 1.7]{BS23}}]
Let $(X, \ll, \leq, d, \tau)$ be a Lorentzian pre-length space and $I \subset \R$ an interval.
A future-directed causal curve $\gamma \colon I \to X$ is called \emph{geodesic} if
for each $t \in I, $ there exists a neighborhood $[a, b] \subset I$ of $t$ such that $\gamma|_{[a, b]}$
is maximal (cf.~Definition \ref{def:maximal}).
\end{definition}

The following definition was first introduced
by Beem \cite{B76} for Lorentzian manifolds.
We extend this condition to Lorentzian length spaces. 

\begin{definition}\label{def:conditionA}
Let $c \in [0, \infty].$ The Lorentzian pre-length space
$(X, \ll, \leq, d, \tau)$ is said to satisfy \emph{Condition A} if
for any $x, y \in X$ with $x \ll y\ (\text{resp.}\ y \ll x$) and
any future (resp. past) inextendible future-directed (resp. past-directed) causal geodesic $\gamma(t)$
starting at $y$ defined on $[0, c), $
it holds that
$\tau(x, \gamma(t)) \to \infty\ (\text{resp.}\ \tau(\gamma(t), x) \to \infty)$ as $t \to c.$
\end{definition}

Beem \cite{B76} first proved the following result for $C^2$-spacetimes.
We extend this result to globally hyperbolic Lorentzian length spaces.

\begin{theorem}\label{conditionA}
Let $(X, \ll, \leq, d, \tau)$ be a globally hyperbolic Lorentzian length space.
If $(X, \ll, \leq, d, \tau)$ is timelike Cauchy complete, then $(X, \ll, \leq, d, \tau)$ satisfies Condition A.
\end{theorem}

\begin{proof}
Let $x, y \in X$ with $x \ll y$ and let $\gamma(t)$ be a future inextendible, future-directed causal geodesic
starting at $y$, defined on $[0,c)$. We argue by contradiction: if Condition A fails,
we construct a sequence as in Definition \ref{def:timelike} that forces $\gamma$ to be extendible.

Suppose that there exists $B > 0$ such that
$\tau(x, \gamma(t)) \leq B$ for all $t \in [0, c).$
Let $(t_n)_n$ be an increasing sequence of positive numbers such that $t_n \to c$ as $n \to \infty.$
Then it follows from Fact \ref{pushup} and Fact \ref{geodesic} that
there exist positive numbers $(c_n)_n$ and maximal future-directed causal curves $(G_n(s))_n$
defined on $[0, c_n]$ such that
$G_n(0) = x$ and $G_n(c_n) = \gamma(t_n)$. 

We will construct a timelike Cauchy sequence $(z_n)_n.$
By the continuity of $d$, we can take $z_1 \in G_1((0, c_1))$ satisfying $d(z_1, \gamma(t_1)) < 1.$
We will show that if $z_n$ is taken,
we can take $z_{n+1}$ in $G_{n+1}((0, c_{n+1}))$
satisfying $d(z_{n+1}, \gamma(t_{n+1})) < \frac{1}{n+1}$ and $z_n \ll z_{n+1}$. 
We only need to show that $z_{n+1}$ can be taken satisfying $z_n \ll z_{n+1}$. 

First, we show that $z_n$ can be taken satisfying $z_n \ll \gamma(t_n).$
Suppose that $G_n(s) \not\ll \gamma(t_n)$ for all $s \in (0, c_n).$
Then, by Fact \ref{finite} (i) and Definition \ref{LLS} (iii),
we have $\tau(G_n(s), \gamma(t_n)) = 0$ for all $s \in (0, c_n).$
By Fact \ref{cont} (ii), we obtain
\begin{align}
0 = \lim_{s \to +0} \tau(G_n(s), \gamma(t_n)) = \tau(x, \gamma(t_n)).
\end{align}
Thus, we see that $x \not\ll \gamma(t_n).$
However, since $x \ll y \leq \gamma(t_n), $
by Fact \ref{pushup},
we obtain $x \ll \gamma(t_n)$, a contradiction.
Thus we can take $z_n$ satisfying $z_n \ll \gamma(t_n).$

Next, we show that $z_{n+1}$ can be taken satisfying $z_n \ll z_{n+1}$. 
Suppose that $z_n \not\ll G_{n+1}(s)$ for all $s \in (0, c_{n+1}).$
Then, by Fact \ref{finite} (i) and Definition \ref{LLS} (iii),
we have
$\tau(z_n, G_{n+1}(s)) = 0$ for all $s \in (0, c_{n+1}).$
Thus, by Definition \ref{LLS} (ii) and Fact \ref{cont} (ii), we obtain
\begin{align}
0 &\leq \tau(z_n, \gamma(t_n)) + \tau(\gamma(t_n), \gamma(t_{n+1})) \\
&\leq \tau(z_n, \gamma(t_{n+1})) = \lim_{s \to c_{n+1}}\tau(z_n, G_{n+1}(s)) = 0,
\end{align}
which implies $\tau(z_n, \gamma(t_n)) = 0.$ This contradicts $z_n \ll \gamma(t_n)$ 
and thus we can take $z_{n+1}$ satisfying $z_n \ll z_{n+1}$. 

Since $x \ll z_n \ll z_{n+1}$, we see that the sequence $(\tau(x, z_n))_n$ is increasing.
By the assumption of contradiction that $(\tau(x, z_n))_n$ is bounded above by $B$,
it follows that $(\tau(x, z_n))_n$ converges to some non-negative constant $b.$
We obtain
\begin{align*}
\tau(z_n, z_{n+m}) &\leq \tau(x, z_{n+m}) - \tau(x, z_n)\\
&\leq b - \tau(x, z_n) \to 0\ (n \to \infty).
\end{align*}
Thus, by the timelike Cauchy completeness, the sequence $(z_n)_n$ converges
to some $z_\infty \in X.$
Since $d(z_n, \gamma(t_n)) \leq \frac{1}{n} \to 0$ as $n \to \infty, $
we see that $\gamma(t_n) \to z_\infty.$
This contradicts the future inextendibility of $\gamma$.
\end{proof}

The following theorem is obtained by combining 
Theorem \ref{timelikeCauchycomplete} and Theorem \ref{conditionA}.

\begin{theorem}\label{thm:LLS}
Let $(X, \ll, \leq, d, \tau)$ be a globally hyperbolic Lorentzian length space.
\begin{enumerate}
\item If $(X, \ll, \leq, d, \tau)$ is finitely compact, then $(X, \ll, \leq, d, \tau)$ is timelike Cauchy complete. 
\item If $(X, \ll, \leq, d, \tau)$ is timelike Cauchy complete, then $(X, \ll, \leq, d, \tau)$ satisfies Condition A.
\end{enumerate}
\end{theorem}

As shown above, in globally hyperbolic Lorentzian length spaces, 
finite compactness implies timelike Cauchy completeness, 
and timelike Cauchy completeness implies Condition A.  

However, to derive finite compactness from Condition A, 
one needs to analyze in detail the behavior of geodesics carefully.  
For instance, there exists a spacetime with a $C^1$ Lorentzian metric, 
which is an example of a Lorentzian length space, 
in which geodesic uniqueness fails \cite[Example~3.2]{KOSS22}. 
In what follows, we focus on the particular class of Lorentzian length spaces, 
namely $C^1$-spacetimes, and examine additional assumptions 
under which finite compactness can be recovered from Condition A.

\section{Geodesics in $C^1$-spacetimes}

In this section, we investigate additional assumptions under 
which finite compactness can be derived from Condition A in $C^1$-spacetimes, 
which form a particular instance of Lorentzian length spaces.

To this end, we first analyze the properties of geodesics in $C^1$-spacetimes.
In a $C^1$-spacetime, the Christoffel symbols are continuous, 
and hence geodesics can be defined 
in the same way as in the smooth case (cf.~\cite[Definition~2.1]{G20}).

\begin{definition}
    Let $(M, g)$ be an $n$-dimensional $C^1$-spacetime. 
    A $C^2$-curve $\gamma \colon I \to M$ is called a \emph{geodesic} if it satisfies 
    $\nabla_{\dot{\gamma}}\dot{\gamma} = 0$, i.e., the geodesic equation
    \begin{equation}
        \frac{d^2\gamma^k}{dt^2} + 
        \sum_{i, j=0}^{n-1} \Gamma_{ij}^k \frac{d\gamma^i}{dt} \frac{d\gamma^j}{dt} = 0,
        \qquad (k=0, \dots, n-1).
    \end{equation}
    Here $\Gamma^k_{ij}$ denote the Christoffel symbols,
    \begin{equation}
        \Gamma^k_{ij} = \frac{1}{2} \sum_{l=0}^{n-1}g^{kl}
        \bigl(\partial_i g_{jl} + \partial_j g_{il} - \partial_l g_{ij} \bigr),
    \end{equation}
    which are continuous functions since the metric coefficients $g_{ij}$ are $C^1$.
\end{definition}

The following still holds for geodesics even in $C^1$-spacetimes. 

\begin{fact}[{\cite[Corollary 2.4]{G20}}]\label{fact:causality}
    In a $C^1$-spacetime, any geodesic does not change 
    the causality. 
\end{fact}

In $C^1$-spacetime, or more generally $C^0$-spacetime, 
the curve length and Lorentz distance are defined in the usual way.

\begin{definition}[{\cite[Definition 4.1]{BEE96}}]\label{def:Lorentzian-distance}
Let $(M,g)$ be a $C^0$-spacetime and let $p,q\in M$. 
A future-directed causal curve $\gamma\colon[a,b]\to M$ means 
a locally Lipschitz curve whose tangent vector is future-directed and causal almost everywhere. 
Its $g$-length is
\begin{equation}
L_g(\gamma)\coloneqq \int_a^b \sqrt{-\,g(\dot\gamma(t),\dot\gamma(t))}\,dt,
\end{equation}
which is well-defined and independent of the reparametrization.  
Let $\Omega_{p,q}$ be the set of all future-directed causal curves from $p$ to $q$.  
The \emph{Lorentzian distance} is defined by 
\begin{equation}
T(p,q)\coloneqq
\begin{cases}
\displaystyle \sup_{\gamma\in\Omega_{p,q}} L_g(\gamma), & \Omega_{p,q}\neq\emptyset,\\
0, & \Omega_{p,q}=\emptyset.
\end{cases}
\end{equation}
\end{definition}

\begin{definition}
Let $(M,g)$ be a $C^1$-spacetime. 
A future-directed causal curve $\gamma\colon[a,b]\to M$ is called \emph{maximal} if
\[
L_g(\gamma)\ \ge\ L_g(\lambda)
\]
for every future-directed causal curve $\lambda$ with the same endpoints as $\gamma$.
\end{definition}

In a $C^1$-spacetime, the geodesic equation is merely continuous,   
so the geodesic ODE has local solutions, but uniqueness may fail.
Indeed, \cite[Example 3.2]{KOSS22} exhibits a 
$C^1$-spacetime in which maximal timelike geodesics and null geodesics branch. 
Consequently, to guarantee the uniqueness of causal geodesics, 
we impose causally non-branching which is a stronger condition 
than maximally causally non-branching (cf.~\cite[Definition 3.1]{KOSS22}). 
We also show that this non-branching hypothesis is actually satisfied by a 
class of $C^1$ low-regularity spacetimes treated in 
Appendix~A; see Theorem~\ref{prop:CNB-non-intertwining}.

\begin{definition}[cf.~{\cite[Definition 3.1]{KOSS22}}]\label{def:non-branching}
    Let $(M, g)$ be a $C^1$-spacetime. 
    A geodesic $\gamma \colon [a, b] \to M$ is said to 
    \emph{branch} at $t_0 \in (a, b)$ 
    if there exists $\epsilon > 0$ and a geodesic 
    $\sigma \colon (t_0-\epsilon, t_0+\epsilon) \to M$ such that 
    $\gamma|_{(t_0-\epsilon, t_0]} = \sigma|_{(t_0-\epsilon, t_0]}$ and 
    $\gamma|_{(t_0, t_0 + \epsilon)} \cap \sigma|_{(t_0, t_0 + \epsilon]} = \emptyset$. 
    $(M, g)$ is called \emph{maximally causally non-branching (MCNB)} if 
    no maximal causal geodesic branches in this sense. 
    $(M, g)$ is called \emph{causally non-branching} if 
    no causal geodesic branches in this sense.
\end{definition}

We define the following sufficient condition ensuring 
that causal geodesics do not branch. 

\begin{definition}\label{def:non-intertwining}
Let $(M,g)$ be a $C^1$-spacetime and $p \in M$. 
We say that $(M, g)$ is \emph{non-intertwining at $p$} if there exists a neighborhood $U$ of $p$ 
such that, for any $\varepsilon > 0$ and any causal geodesics 
$\gamma_1 \colon [0, \varepsilon] \to U$ and 
$\gamma_2 \colon [0, \varepsilon] \to U$ starting at $p$ with the same initial velocity $v$,  
it holds that $\gamma_1(t) = \gamma_2(t)$ for all $t \in [0, \varepsilon]$ or 
$\gamma_1((0, \varepsilon]) \cap \gamma_2((0, \varepsilon])$ has no accumulation point.  
We say $(M,g)$ is \emph{non-intertwining} if it is non-intertwining at every $p\in M$. 
\end{definition}

The causally non-branching and non-intertwining conditions ensure that 
any causal geodesic with the same initial condition is unique. 

\begin{lemma}\label{lem:uniq-from-proper}
Let $(M,g)$ be a causally non-branching and non-intertwining $C^1$-spacetime.  
Suppose $\gamma_1,\gamma_2 \colon [0,1]\to M$ are causal geodesics with 
the same initial condition: 
\[
\gamma_1(0)=\gamma_2(0)=p,\qquad \dot\gamma_1(0)=\dot\gamma_2(0)=v.
\]
Then $\gamma_1=\gamma_2$ on $[0,1]$.
\end{lemma}

\begin{proof}
First, we show that $\gamma_1$ and $\gamma_2$ coincide on 
a small interval $[0, t_1]$ where $t_1 > 0$. 
Suppose, for the sake of contradiction, that $\gamma_1$ and $\gamma_2$ 
do not coincide on any neighborhood of $0$. 
The non-intertwining condition implies that there exists $\varepsilon > 0$ such that 
$\gamma_1((0, \varepsilon]) \cap \gamma_2((0, \varepsilon])$ has no accumulation point. 
This guarantees the existence of a possibly smaller $\eta \in (0, \varepsilon]$ such that 
$\gamma_1(t) \neq \gamma_2(t)$ for all $t \in (0, \eta]$.

Let $\gamma_- \colon (-\delta, 0] \to M$ be a past-directed causal geodesic 
with $\gamma_-(0) = p$ and $\dot\gamma_-(0) = v$, 
obtained from local existence for the geodesic ODE. 
Define concatenations
\begin{equation}
\tilde\gamma_i(t)=
\begin{cases}
\gamma_-(t), & t\in(-\delta,0],\\
\gamma_i(t), & t\in[0,\eta],
\end{cases}\qquad i=1,2.
\end{equation}
Each $\tilde\gamma_i$ is a causal geodesic on $(-\delta,\eta]$, the two coincide on $(-\delta,0]$, 
but they differ on $(0, \eta]$. 
Thus, we have a causal geodesic that branches at the interior point $t=0$, 
contradicting causal non-branching. 

Therefore, the assumption is false, and $\gamma_1$ and $\gamma_2$ must coincide on 
some small interval. By a standard connectedness argument, this local uniqueness extends to the whole interval $[0, 1]$. 
\end{proof}

Since the causal geodesic with a given initial condition is unique in the above situation, 
we can define the causal exponential map as follows.

\begin{definition}\label{def:exp}
    Let $(M, g)$ be a causally non-branching and non-intertwining $C^1$-spacetime and 
    let $p \in M$. 
    Then, we can define the \emph{domain of the causal exponential map} 
    at $p$ by 
    \begin{equation}
        \mathcal{E}_p \coloneqq 
        \{v \in T_pM ; g_p(v, v) \le 0, \text{$\gamma_v$ is defined on an interval containing $[0,1]$}\}
    \end{equation}
    and the \emph{causal exponential map} $\exp_p \colon \mathcal{E}_p \to M$ at $p$ by 
    \begin{equation}
        \exp_p(v) \coloneqq \gamma_v(1)
    \end{equation}
    where $\gamma_v$ is the causal geodesic with 
    $\gamma(0) = p$ and $\dot{\gamma}(0) = v$. 
\end{definition}

It is worth noting that for $C^{1,1}$-spacetimes, the exponential map is 
a bi-Lipschitz homeomorphism locally around any point \cite[Theorem 2.1]{KSS14}. 
On the other hand, for the $C^1$-spacetimes considered here, 
we establish the continuity of the causal exponential map.
This follows from standard ODE theory, and the proof proceeds 
along the lines of \cite[Corollary 2.6]{KOSS22}, adapted to the present context. 

\begin{proposition}\label{prop:continuous}
Let $(M, g)$ be a causally non-branching, non-intertwining $C^1$-spacetime. 
Then the causal exponential map $\exp_p \colon \mathcal{E}_p \to M$ at $p$ is 
continuous for every $p \in M$. 
\end{proposition}

\begin{proof}
It suffices to consider the case where the geodesic is contained 
in a single coordinate chart. The result extends to the general case 
by a finite covering argument, noting that the convergence of position and velocity 
at the end of one chart implies the convergence of initial data for the next. 

Let $\{v_k\}_{k \in \N} \subset \mathcal{E}_p$ be a sequence 
of causal vectors converging to $v \in \mathcal{E}_p$.
In a local coordinate $(U, \phi)$, the geodesic equation with initial point 
$\tilde{p} = (\tilde{p}^1, \dots, \tilde{p}^n)$ and initial velocity 
$\tilde{v} = (\tilde{v}^1, \dots, \tilde{v}^n)$ 
is rewritten as first order differential equations. 
Let $\tilde{y}(t) = (\tilde{y}^1(t), \dots, \tilde{y}^{2n}(t))^T$. The equation is as follows:
\begin{equation} \label{eq:ode_system} \tag{$*$}
\frac{d}{dt}
\begin{pmatrix}
\tilde{y}^1(t) \\
\vdots \\
\tilde{y}^n(t) \\
\tilde{y}^{n+1}(t) \\
\vdots \\
\tilde{y}^{2n}(t)
\end{pmatrix}
=
\begin{pmatrix}
\tilde{y}^{n+1}(t) \\
\vdots \\
\tilde{y}^{2n}(t) \\
-\Gamma^1_{ij}(\tilde{y}^1, \dots, \tilde{y}^n) \tilde{y}^{n+i} \tilde{y}^{n+j} \\
\vdots \\
-\Gamma^n_{ij}(\tilde{y}^1, \dots, \tilde{y}^n) \tilde{y}^{n+i} \tilde{y}^{n+j}
\end{pmatrix}
\end{equation}
with initial conditions:
\begin{align*}
\tilde{y}^l(0) &= \tilde{p}^l \quad (1 \le l \le n), \\
\tilde{y}^l(0) &= \tilde{v}^{l-n} \quad (n+1 \le l \le 2n).
\end{align*}

Define the function $F$ by the right-hand side of \eqref{eq:ode_system}:
\[
F(t, \tilde{y}^1, \dots, \tilde{y}^{2n}) \coloneqq (\text{RHS of } \eqref{eq:ode_system}).
\]
Then $F$ is continuous since the metric $g$ is $C^1$. 

Let $Y_k$ be the solution of \eqref{eq:ode_system} with the initial conditions corresponding to 
$p = (p^1, \dots, p^n)$ and $v_k = (v_k^1, \dots, v_k^n)$.
Since $v_k \to v$, the initial data converges as $(p, v_k) \to (p, v)$.
By \cite[Proposition 2.5]{KOSS22}, 
for any subsequence of $\{Y_k\}$, there exists a further subsequence 
$\{Y_{k_m}\}$ and a solution $\tilde{Y}$ of \eqref{eq:ode_system} with the initial condition 
corresponding to $(p, v)$ such that
\[
Y_{k_m} \to \tilde{Y} \qquad \text{in } C^1([0, 1], \R^{2n}).
\]

Let $Y$ be the solution of \eqref{eq:ode_system} corresponding to 
$p = (p^1, \dots, p^n)$ and $v = (v^1, \dots, v^n)$. 
By Lemma 4.7, any two causal geodesics with the same initial conditions coincide.  
Therefore, the limit solution $\tilde{Y}$ obtained from the subsequence 
must coincide with the solution $Y$, which defines the 
causal exponential map $\exp_p$.

Since every convergent subsequence converges to the same unique limit $Y$, 
the entire sequence $\{Y_k\}$ converges to $Y$ in $C^1([0, 1], \R^{2n})$.
In particular, evaluating at $t = 1$, the position components converge:
\[
\lim_{k \to \infty} (Y_k^1(1), \dots, Y_k^n(1)) = (Y^1(1), \dots, Y^n(1)).
\]
This means $\lim_{k \to \infty} \exp_p(v_k) = \exp_p(v)$, proving 
the continuity of the exponential map.
\end{proof}

We adopt the definition of global hyperbolicity for $C^1$-spacetimes 
using non-totally imprisoning instead of causality.

\begin{definition}[{\cite[Definition 2.6 and Definition 3.1]{S16}}]
    A $C^1$-spacetime is called \emph{globally hyperbolic} if 
    \begin{enumerate}
        \item non-totally imprisoning, i.e., no future or past inextendible causal curve is 
        contained in a compact set,  
        \item the causal diamond $J(p,q) \coloneqq J^+(p) \cap J^-(q)$ is 
        compact for all $p, q \in M$ .
    \end{enumerate}
\end{definition}

We show that the Lorentzian distance 
is strictly increasing along causal geodesics when one of its arguments is slightly 
shifted away from the starting point.

\begin{lemma}\label{lem:increasing}
    Let $(M, g)$ be a globally hyperbolic $C^1$-spacetime, 
    $p, q \in M$ with $p \ll q$ and 
    $x \colon I \to M$ be a causal geodesic starting at $q$. 
    Then the function $T(p, x(t))$ is strictly monotonically increasing.  
\end{lemma}

\begin{proof}
    Let $t_1, t_2 \in I$ with $t_1 < t_2$. 
    The reverse triangle inequality implies that 
    \begin{equation}
        T(p, x(t_1)) + T(x(t_1), x(t_2)) \leq  T(p, x(t_2)). 
    \end{equation} 
    By Fact \ref{fact:causality}, it suffices to consider the case 
    where $x$ is timelike or null. 
    If $x$ is timelike, then it follows that 
    $T(x(t_1), x(t_2)) > 0$, which means that 
    $T(p, x(t))$ is strictly monotonically increasing.

    If $x$ is null, assume there exist $t_1, t_2 \in I$ with $t_1 < t_2$ 
    such that $T(p, x(t_1)) = T(p, x(t_2))$. 
    By \cite[Proposition 2.13]{G20}, there exist timelike geodesics 
    $\gamma_1$ and $\gamma_2$ such that 
    $L_g(\gamma_1) = T(p, x(t_1))$ and $L_g(\gamma_2) = T(p, x(t_2))$. 
    Then the curve $\gamma$ obtained by concatenating $\gamma_1$ and $x|_{[t_1, t_2]}$ 
    is a maximal curve from $p$ to $x(t_2)$. 
    By \cite[Theorem 3.3]{SS21}, 
    $\gamma$ is a causal geodesic, any maximal causal curve is a causal geodesic; 
    hence $\gamma$ is a geodesic which changes causal character along $x|_{[t_1, t_2]}$,  
    contradicting Fact \ref{fact:causality}.  
\end{proof}

Finally, we show that Condition A implies finite compactness for a $C^1$-spacetime.  

\begin{theorem}\label{th:finite-compactness}
    Let $(M, g)$ be a connected, globally hyperbolic, causally non-branching, 
    and non-intertwining $C^1$-spacetime. 
    If $(M, g)$ satisfies Condition A, then $(M, g)$ is finitely compact.  
\end{theorem}

\begin{proof}
    We prove the result by adapting \cite[Theorem 5]{B76} 
    to our setting. 
    By \cite[Proposition 5.9, Theorem 5.6]{S16}, $(M, g)$ is strongly causal 
    as defined for spacetimes  
    and also causally plain by \cite[Corollary 1.17]{CG12}. 
    Thus, by \cite[Theorem 5.12]{KS18}, 
    the 5-tuple $(M, \ll, \leq , d^h, T)$ defined via locally Lipschitz curves 
    (cf.~\cite[Section 5.1]{KS18}) is a Lorentzian length space.  
    The strong causality implies the non-totally imprisoning property 
    as defined for Lorentzian length spaces 
    (cf.~\cite[Theorem 3.26 (iv), (iii)]{KS18}), 
    and hence global hyperbolicity.  

    Let $p, q \in M$ with $p \ll q$ and $B > 0$.  
    From Lemma \ref{compact}, it suffices to show that 
    $K_1 = \{x \in M ; p \ll q \leq x, T(p, x) \leq B\}$ 
    is compact. 
    By Fact \ref{closed} and Fact \ref{cont} (ii), 
    the set $K_1 = J^+(q) \cap \{x \in M; T(p, x) \leq B\}$ is closed. 

    Let $(z_n)_n$ be a sequence of points in $K_1$. 
    If for every $N \in \N$, there exists $n \geq N$ such that $z_n = q$, 
    then there is a subsequence that converges to $q \in K_1$. 
    Otherwise, we may assume $z_n \not= q$ for all $n \in \N$. 

    By \cite[Proposition 2.13]{G20}, there exists a maximal causal geodesic 
    $G_n \colon [0,1] \to M$ from $q$ to $z_n$.  
    Let $X_n  \coloneqq \dot{G}_n(0)$, then since 
    $\exp_q$ is defined, we have $G_n(t) = \exp_q(tX_n)$.  
    Let $L_n$ be the direction 
    containing $X_n$ (cf.~\cite[Section 2]{B76}).  
    By the reverse triangle inequality, 
    we have $T(p, q) \leq T(p, z_n) \leq B$ and 
    the set of future-pointing (causal) directions at $q$ is 
    a compact set in the space of directions, thus, 
    passing to a subsequence and relabeling it again by $n$, we have
    $L_n \to L_0\ (n \to \infty)$ where $L_0$ is the future-pointing direction at $q$ and 
    $T(p, z_n) \to b\ (n \to \infty)$ where $T(p, q) \leq b \leq B$. 

    Let $x(t) \colon [0, c) \to M$ be a future inextendible maximal causal geodesic starting at $q$ 
    with $\dot{x}(0) \in L_0$. 
    By Condition~A, $T(p,x(t)) \to \infty$ as $t \to c$, and by Lemma~\ref{lem:increasing}
    the map $t \mapsto T(p,x(t))$ is strictly increasing. 
    By Fact~\ref{cont}, $T$ is continuous and there is a unique $t_0\in(0,c)$ and $X_0 \in L_0$ 
    satisfying $T(p,x(t_0))=b$ and $\exp_q(X_0)=x(t_0)$.

    Let $\norm{\ }$ be a Euclidean inner product on $T_qM$. 
    For each direction $L$ at $q$ let $u(L)$ be the unique future-pointing causal vector in $L$ 
    with $\norm{u(L)} = 1$.
    Put $u_n \coloneqq u(L_n)$ and $u_0 \coloneqq u(L_0)$. 
    Since $L_n \to L_0\ (n \to \infty)$, we have $u_n \to u_0\ (n \to \infty)$.
    For each $n$ write
    \begin{equation}
        X_n = s_nu_n,\ s_n \coloneqq \norm{X_n} > 0,
    \end{equation}
    and similarly $X_0 = s_0u_0$ with $s_0 > 0$ characterized by $\exp_q(s_0u_0)=x(t_0)$.
    Define
    \begin{equation}
        F(u, s) \coloneqq T\left(p,\exp_q(su)\right).
    \end{equation}
    By Proposition \ref{prop:continuous} and the continuity of $T$, the map $F$ is continuous 
    in $(u, s)$ on 
    the domain of exponential map and by Lemma~\ref{lem:increasing}, for each fixed $u$ the function 
    $s \mapsto F(u,s)$ is strictly increasing.

    We have $F(u_n,s_n) = T(p,z_n) \to b\ (n \to \infty)$ and $F(u_0, s_0) = T(p, x(t_0)) = b$.
    Suppose that $s_n \not\to s_0$.
    Then there exists $\varepsilon > 0$ and a subsequence (still denoted $n$) 
    with $\abs{s_n-s_0} \geq \varepsilon$ for all $n$. 
    Choose sufficiently small $\eta > 0$ and take a compact set 
    $C \coloneqq \{u \in T_qM \colon \norm{u-u_0} \leq \eta, \norm{u}=1\}$.
    Define the continuous functions on $C$ by 
    \begin{equation}
      G_+(u) \coloneqq F(u,s_0+\varepsilon)-F(u,s_0),\qquad
      G_-(u) \coloneqq F(u,s_0)-F(u,s_0-\varepsilon).
    \end{equation}
    By strict monotonicity in $s$ and the compactness of $C$, we have 
    \begin{equation}
      \delta_\varepsilon \coloneqq \min_{u \in C}G_+(u)>0,\qquad
      \tilde\delta_\varepsilon \coloneqq \min_{u \in C}G_-(u)>0.
    \end{equation}
    Passing to a further subsequence, 
    the sequence $(s_n)_n$ satisfies either $s_n \geq s_0 + \varepsilon$ for all $n$, 
    or $s_n \leq s_0 - \varepsilon$ for all $n$.
    In the first case,
    \begin{equation}
       F(u_n,s_n) \geq F(u_n,s_0+\varepsilon) \geq F(u_n,s_0) + \delta_\varepsilon,
    \end{equation}
    while in the second,
    \begin{equation}
       F(u_n,s_n)\ \leq F(u_n,s_0-\varepsilon) \leq F(u_n,s_0) - \tilde\delta_\varepsilon.
    \end{equation}
    The continuity of $F$ and  $u_n \to u_0$ yield $F(u_n, s_0) \to F(u_0, s_0) = b\ (n \to \infty)$, 
    whereas we also have $F(u_n, s_n) \to b$ as $n \to \infty$, a contradiction.
    Hence $s_n \to s_0$, and with $u_n \to u_0$ we conclude
    $X_n = s_nu_n \to s_0u_0 = X_0$. 

    Finally, by Proposition \ref{prop:continuous}, we obtain
    $z_n=\exp_q(X_n)\to \exp_q(X_0)=x(t_0)\in K_1$. 
    Therefore every sequence in $K_1$ has a convergent subsequence with limit in $K_1$,
    so $K_1$ is compact. This proves finite compactness.
\end{proof}

The following theorem is obtained by combining 
Theorem \ref{timelikeCauchycomplete}, Theorem \ref{conditionA}, 
and Theorem \ref{th:finite-compactness}.

\begin{theorem}\label{thm:main}
Let $(M, g)$ be a connected, globally hyperbolic, causally non-branching, 
and non-intertwining $C^1$-spacetime.  
Then the following three conditions are equivalent:
\begin{enumerate}
    \item $M$ is finitely compact;
    \item $M$ is timelike Cauchy complete;
    \item $M$ satisfies Condition A.
\end{enumerate}
\end{theorem}

\appendix

\section{Examples of $C^1$-Spacetimes with geodesic uniqueness}

In this appendix, we introduce examples of $C^1$-spacetimes 
where the uniqueness of causal geodesics holds 
and the causal exponential map as in Definition \ref{def:exp} can be defined.

We work on $M=\R^2$ with global coordinates $(x^0,x^1)=(t,x)$.  
The metric has the form
\begin{equation}\label{eq:metric-ab-numbered}
  g= -a(x^0)(dx^0)^2 + b(x^0)(dx^1)^2,
\end{equation}
where $a,b\in C^1(\R)$ are strictly positive functions of the time coordinate $x^0$.
A dot denotes differentiation with respect to an affine parameter $s$ along a curve,
and a prime on $a,b$ denotes derivative with respect to $x^0$, 
e.g.\ $a^\prime(x^0)=\frac{da}{dx^0}(x^0)$.

\begin{proposition}\label{prop:ab-uniqueness-numbered}
Let $g$ be as in \eqref{eq:metric-ab-numbered}. 
Fix an initial point $p \coloneqq (x^0_0,x^1_0) \in \R^2$ and an initial velocity 
$v \coloneqq (\tau_0,\xi_0) \in T_{p}\R^2$. 
If $\tau_0\neq 0$, there exist $\delta>0$ and a unique geodesic
\[
\gamma(s)=(x^0(s),x^1(s))\colon (-\delta,\delta)\to\R^2
\]
satisfying the initial conditions $\gamma(0)=p$ and $\dot\gamma(0)=v$. 
In particular, every non-trivial causal geodesic satisfies $\tau_0\neq0$, 
hence the above local uniqueness applies to all causal geodesics.
\end{proposition}

\begin{proof}
Without loss of generality, we assume the initial data are future-directed, 
i.e. $\tau_0=\dot x^0(0)>0$.
A direct computation gives the only non-trivial Christoffel symbols:
  \begin{equation}\label{eq:Gamma-numbered}
  \Gamma^{0}_{00}=\frac{a^\prime}{2a},\qquad
  \Gamma^{0}_{11}=\frac{b^\prime}{2a},\qquad
  \Gamma^{1}_{01}=\Gamma^{1}_{10}=\frac{b^\prime}{2b},
\end{equation}
all others vanish. Thus, the geodesic equations are 
\begin{equation}\label{eq:geo-numbered}
  \ddot x^0+\frac{a^\prime}{2a}\,(\dot x^0)^2+\frac{b^\prime}{2a}\,(\dot x^1)^2=0,
  \qquad
  \ddot x^1+\frac{b^\prime}{b}\,\dot x^0\,\dot x^1=0.
\end{equation}

The $x^1$-equation in \eqref{eq:geo-numbered} gives 
\begin{equation}
\ddot x^1+\frac{b^\prime}{b}\,\dot x^0\,\dot x^1
=\frac{1}{b}\frac{d}{ds}\!\bigl(b(x^0)\,\dot x^1\bigr)=0, 
\end{equation}
and thus, we have 
\begin{equation}\label{eq:constant}
  \,b(x^0)\dot x^1=\kappa\in\R.
\end{equation}
Define the constant squared speed
\begin{equation}
\varepsilon \coloneqq g(\dot\gamma,\dot\gamma)=
-a(x^0)\,(\dot x^0)^2+b(x^0)\,(\dot x^1)^2.
\end{equation}
Combining with \eqref{eq:constant} yields
\begin{equation}\label{eq:t-first-order-numbered}
  a(x^0)\,(\dot x^0)^2=\frac{\kappa^2}{b(x^0)}-\varepsilon.
\end{equation}
Evaluating at $s=0$ gives $\kappa=b(x^0_0)\xi_0$ and 
$\varepsilon=-a(x^0_0)\tau_0^2+b(x^0_0)\xi_0^2$, and hence
\begin{equation}\label{eq:t-first-order-IC-numbered}
  a(x^0)(\dot x^0)^2
  =\kappa^{2}\left(\frac1{b(x^0)}-\frac1{b(x^0_0)}\right)+a(x^0_0)\tau_0^{2}.
\end{equation}
If $\tau_0\neq0$, the left-hand side is larger than $0$ at $x^0=x^0_0$ 
and stays positive near $x^0_0$ by continuity. 
Using the future-directed assumption $\tau_0>0$,
we fix the positive branch
\begin{equation}
\dot x^0= \sqrt{\frac{\kappa^{2}/b(x^0)-\varepsilon}{a(x^0)}} > 0.
\end{equation}
Separating variables yields the implicit integral
\begin{equation}\label{eq:inverse}
s-s_0=\int_{x^0_0}^{x^0(s)}
  \frac{\sqrt{a(u)}}{\sqrt{\kappa^{2}/b(u)-\varepsilon}}\,du.
\end{equation}
The integrand is continuous and strictly positive in a neighborhood of $x^0_0$, 
so the right-hand side is strictly increasing and continuous in $x^0$. 
By the inverse function theorem for monotone continuous functions, 
\eqref{eq:inverse} determines $x^0$ as a unique function of $s$. 
From \eqref{eq:constant},  
\begin{equation}
x^1(s)=x^1_0+\int_{s_0}^{s}\frac{\kappa}{b(x^0(u))}du, 
\end{equation}
which is uniquely determined once $x^0$ is. 
This proves local uniqueness on some $(-\delta,\delta)$.

Finally, if a non-trivial causal geodesic satisfies $\dot x^0(0)=\tau_0=0$ at the initial point,
the causality at $s=0$,
\begin{equation}
-a(x^0_0)\,(\dot x^0(0))^2 + b(x^0_0)\,(\dot x^1(0))^2 \le 0,
\end{equation}
forces $\dot x^1(0)=\xi_0=0$, a contradiction.
\end{proof}

\begin{theorem}\label{prop:CNB-non-intertwining}
Under the metric \eqref{eq:metric-ab-numbered}, 
$(M, g)$ is causally non-branching and non-intertwining. 
\end{theorem}

\begin{proof}
  Let $\gamma$ be a non-trivial causal geodesic. 
Proposition~\ref{prop:ab-uniqueness-numbered} establishes that for any point on $\gamma$, 
the solution to the geodesic equation is locally unique given the initial position and velocity. 

This local uniqueness immediately implies the non-intertwining condition. 
Indeed, for any $p \in M$ and causal vector $v \in T_pM$, any two geodesics starting 
with $(p, v)$ must coincide on some time interval $[0, \delta)$, 
which prevents them from intertwining.

Furthermore, since this local uniqueness holds at every point along any causal geodesic, 
branching cannot occur. 
Therefore, $(M, g)$ is causally non-branching.
\end{proof}

\begin{remark}
  Theorem~\ref{prop:CNB-non-intertwining} confirms that the spacetimes 
  defined by \eqref{eq:metric-ab-numbered} satisfy the assumptions required 
  in Definition~\ref{def:exp}. 
  Consequently, the causal exponential map is well-defined, 
  and its continuity is guaranteed by Proposition~\ref{prop:continuous}. 
\end{remark}

As discussed above, any causally non-branching and non-intertwining $C^1$-spacetime 
admits a well-defined continuous causal exponential map, whereas for $C^{1,1}$-spacetimes, 
this map is a local bi-Lipschitz homeomorphism \cite[Theorem 2.1]{KSS14}. 
Moreover, the $C^1$-Lorentzian splitting theorem \cite[Theorem 2]{BGMOS25} establishes that 
a certain class of $C^1$-spacetimes is $C^2$-isometric to the product of $\R$ and 
a Riemannian manifold equipped with a $C^1$-metric. In the following theorem, 
we prove that for this specific class of spacetimes, 
the causal exponential map is $C^2$ on the future causal cone.

\begin{theorem}\label{th:abstract-ab-class}
Let $M=\R^2$ with coordinates $(t,x)$ and metric
\[
g=-a(t)\,dt^2+dx^2,
\]
where $a \in C^1(\R)$ satisfies:
\begin{enumerate}
\item[(H1)] $a(t)\ge \alpha>0$ for all $t\in\R$,
\item[(H2)] $\displaystyle \int_{t_0}^{+\infty}\sqrt{a(u)}\,du=+\infty$ for some $t_0\in\R$.
\end{enumerate}
Then for every $q\in M$, the causal exponential map $\exp_q$ is $C^2$ on the cone of 
future-directed nonzero causal vectors in $T_qM$.
\end{theorem}

\begin{proof}
Define a new time coordinate by
\[
\tau = T(t) \coloneqq \int_{t_*}^{t}\sqrt{a(u)}du ,
\]
with fixed $t_*\in\R$. Since $a\in C^1$ and $a\ge \alpha>0$ by (H1), we have
$\sqrt{a}\in C^1$ and $\sqrt{a}>0$, hence $T\in C^2$ is strictly increasing. Therefore
$T \colon \R\to I$ is a $C^2$-diffeomorphism onto an open interval $I\subset\R$, with inverse
$t=t(\tau)\in C^2$. By (H2), $T(t)\to +\infty$ as $t\to +\infty$, so $I$ is unbounded above:
there exists $A\in[-\infty,\infty)$ such that $I=(A,\infty)$.

Let $\Phi\colon \R^2 \to I\times\R$ be the $C^2$-diffeomorphism $\Phi(t,x)=(\tau,x)=(T(t),x)$. In the
$(\tau, x)$-coordinates, using $dt=d\tau/\sqrt{a(t(\tau))}$, we compute
\[
g \ =\ -a(t)\,dt^2+dx^2 \ =\ -a(t)\,\Bigl(\frac{d\tau}{\sqrt{a(t)}}\Bigr)^2+dx^2
\ =\ -\,d\tau^2+dx^2.
\]
Thus $\Phi$ is an isometry from $(\R^2,g)$ onto $(I\times\R,\eta)$, where
$\eta=-d\tau^2+dx^2$ is the 2-dimensional Minkowski metric.

Fix $q\in M$, and set $Q\coloneqq \Phi(q)\in I\times\R$. For $v\in T_qM$, write
$w\coloneqq d\Phi_q(v)\in T_Q(I\times\R)\simeq\R^2$. If $v$ is future-directed nonzero causal,
then in Minkowski coordinates $w=(w^\tau,w^x)$ satisfies $w^\tau>0$ and $(w^\tau)^2\ge (w^x)^2$.
Along the $\eta$-geodesic with initial data $(Q,w)$, the solution is the straight line
$s\mapsto Q+s\,w$. Since $I=(A,\infty)$ is unbounded above and
$w^\tau>0$, the entire segment $\{\tau(Q)+sw^\tau ; s\in[0,1]\}$ lies in $I$, hence this
geodesic exists on $[0,1]$. Therefore, the Minkowski exponential is globally defined on the
future causal cone and given by
\[
\exp_Q^{\eta}(w) = Q+w .
\]

Since $\Phi$ is an isometry, we have 
\[
\exp_q^{g}(v)\ =\ \Phi^{-1}\!\bigl(\exp_{Q}^{\eta}(d\Phi_q(v))\bigr)
\ =\ \Phi^{-1}\!\bigl(Q+d\Phi_q(v)\bigr).
\]
Here $v\mapsto d\Phi_q(v)$ is smooth, $w\mapsto Q+w$ is smooth, and
$\Phi^{-1}$ is $C^2$. Consequently, the composition $v\mapsto \exp_q^{g}(v)$ is $C^2$ on the
cone of future-directed nonzero causal vectors in $T_qM$.
\end{proof}

\bibliographystyle{abbrv} 
\bibliography{bibliography}

@article {ACS20,
    AUTHOR = {Ak\'{e} Hau, Luis and Cabrera Pacheco, Armando J. and Solis,
              Didier A.},
     TITLE = {On the causal hierarchy of {L}orentzian length spaces},
   JOURNAL = {Classical Quantum Gravity},
  FJOURNAL = {Classical and Quantum Gravity},
    VOLUME = {37},
      YEAR = {2020},
    NUMBER = {21},
     PAGES = {215013, 22},
      ISSN = {0264-9381},
   MRCLASS = {53C50 (83C75)},
  MRNUMBER = {4192594},
MRREVIEWER = {Ettore Minguzzi},
       DOI = {10.1088/1361-6382/abb25f},
       URL = {https://doi.org/10.1088/1361-6382/abb25f},
}

@phdthesis{B23, 
  author  = {Buro, Guillaume},
  title   = {Les g{\'e}od{\'e}siques des m{\'e}triques finsl{\'e}riennes et pseudo-finsl{\'e}riennes faibles de basse r{\'e}gularit{\'e}},
  school  = {{\'E}cole polytechnique f{\'e}d{\'e}rale de Lausanne},
  year    = {2023},
  type    = {Th{\`e}se de doctorat},
  address = {Lausanne}
}

@article {B67,
    AUTHOR = {Busemann, H.},
     TITLE = {Timelike spaces},
   JOURNAL = {Dissertationes Math. (Rozprawy Mat.)},
  FJOURNAL = {Polska Akademia Nauk. Instytut Matematyczny. Dissertationes
              Mathematicae. Rozprawy Matematyczne},
    VOLUME = {53},
      YEAR = {1967},
     PAGES = {52},
      ISSN = {0012-3862},
   MRCLASS = {53.95 (83.00)},
  MRNUMBER = {220238},
MRREVIEWER = {L.\ M.\ Blumenthal},
}

@article {B76,
    AUTHOR = {Beem, John K.},
     TITLE = {Globally hyperbolic space-times which are timelike {C}auchy
              complete},
   JOURNAL = {Gen. Relativity Gravitation},
  FJOURNAL = {General Relativity and Gravitation},
    VOLUME = {7},
      YEAR = {1976},
    NUMBER = {4},
     PAGES = {339--344},
      ISSN = {0001-7701},
   MRCLASS = {83.53 (53C50 83.57)},
  MRNUMBER = {469068},
MRREVIEWER = {Mauro Francaviglia},
       DOI = {10.1007/bf00771104},
       URL = {https://doi.org/10.1007/bf00771104},
}

@book {BBI01,
    AUTHOR = {Burago, Dmitri and Burago, Yuri and Ivanov, Sergei},
     TITLE = {A course in metric geometry},
    SERIES = {Graduate Studies in Mathematics},
    VOLUME = {33},
 PUBLISHER = {American Mathematical Society, Providence, RI},
      YEAR = {2001},
     PAGES = {xiv+415},
      ISBN = {0-8218-2129-6},
   MRCLASS = {53C23},
  MRNUMBER = {1835418},
MRREVIEWER = {Mario\ Bonk},
       DOI = {10.1090/gsm/033},
       URL = {https://doi.org/10.1090/gsm/033},
}

@book {BEE96,
    AUTHOR = {Beem, John K. and Ehrlich, Paul E. and Easley, Kevin L.},
     TITLE = {Global {L}orentzian geometry},
    SERIES = {Monographs and Textbooks in Pure and Applied Mathematics},
    VOLUME = {202},
   EDITION = {Second},
 PUBLISHER = {Marcel Dekker, Inc., New York},
      YEAR = {1996},
     PAGES = {xiv+635},
      ISBN = {0-8247-9324-2},
   MRCLASS = {53C50 (53-02 83-02)},
  MRNUMBER = {1384756},
MRREVIEWER = {Peter\ R.\ Law},
}

@article {BG24,
    AUTHOR = {Burtscher, Annegret and Garc\'ia-Heveling, Leonardo},
     TITLE = {Global hyperbolicity through the eyes of the null distance},
   JOURNAL = {Comm. Math. Phys.},
  FJOURNAL = {Communications in Mathematical Physics},
    VOLUME = {405},
      YEAR = {2024},
    NUMBER = {4},
     PAGES = {Paper No. 90, 35},
      ISSN = {0010-3616,1432-0916},
   MRCLASS = {53C50 (83C05)},
  MRNUMBER = {4719972},
MRREVIEWER = {Mehdi\ Sharifzadeh},
       DOI = {10.1007/s00220-024-04936-5},
       URL = {https://doi.org/10.1007/s00220-024-04936-5},
}

@misc{BGMOS25,
      title={A Lorentzian splitting theorem for continuously differentiable metrics and weights}, 
      author={Mathias Braun and Nicola Gigli and Robert J. McCann and Argam Ohanyan and Clemens Sämann},
      year={2025},
      eprint={2507.06836},
      archivePrefix={arXiv},
      primaryClass={math.DG},
      url={https://arxiv.org/abs/2507.06836}, 
}

@article {BORS23,
    AUTHOR = {Beran, Tobias and Ohanyan, Argam and Rott, Felix and Solis,
              Didier A.},
     TITLE = {The splitting theorem for globally hyperbolic {L}orentzian
              length spaces with non-negative timelike curvature},
   JOURNAL = {Lett. Math. Phys.},
  FJOURNAL = {Letters in Mathematical Physics},
    VOLUME = {113},
      YEAR = {2023},
    NUMBER = {2},
     PAGES = {Paper No. 48, 47},
      ISSN = {0377-9017,1573-0530},
   MRCLASS = {53C50 (53B30 53C23)},
  MRNUMBER = {4579262},
MRREVIEWER = {Benjam\'in\ Olea},
       DOI = {10.1007/s11005-023-01668-w},
       URL = {https://doi.org/10.1007/s11005-023-01668-w},
}

@article {BS23,
    AUTHOR = {Beran, Tobias and S\"{a}mann, Clemens},
     TITLE = {Hyperbolic angles in {L}orentzian length spaces and timelike
              curvature bounds},
   JOURNAL = {J. Lond. Math. Soc. (2)},
  FJOURNAL = {Journal of the London Mathematical Society. Second Series},
    VOLUME = {107},
      YEAR = {2023},
    NUMBER = {5},
     PAGES = {1823--1880},
      ISSN = {0024-6107,1469-7750},
   MRCLASS = {53B30 (28A75 51K10 53C23 53C50 53C80)},
  MRNUMBER = {4585303},
}

@article {CG12,
    AUTHOR = {Chru\'sciel, Piotr T. and Grant, James D. E.},
     TITLE = {On {L}orentzian causality with continuous metrics},
   JOURNAL = {Classical Quantum Gravity},
  FJOURNAL = {Classical and Quantum Gravity},
    VOLUME = {29},
      YEAR = {2012},
    NUMBER = {14},
     PAGES = {145001, 32},
      ISSN = {0264-9381,1361-6382},
   MRCLASS = {53C50 (83A05)},
  MRNUMBER = {2949547},
MRREVIEWER = {Miguel\ S\'anchez},
       DOI = {10.1088/0264-9381/29/14/145001},
       URL = {https://doi.org/10.1088/0264-9381/29/14/145001},
}

@book {dC92,
    AUTHOR = {do Carmo, Manfredo Perdig\~{a}o},
     TITLE = {Riemannian geometry},
    SERIES = {Mathematics: Theory \& Applications},
   EDITION = {Portuguese},
 PUBLISHER = {Birkh\"{a}user Boston, Inc., Boston, MA},
      YEAR = {1992},
     PAGES = {xiv+300},
      ISBN = {0-8176-3490-8},
   MRCLASS = {53-01},
  MRNUMBER = {1138207},
MRREVIEWER = {Bang-yen\ Chen},
       DOI = {10.1007/978-1-4757-2201-7},
       URL = {https://doi.org/10.1007/978-1-4757-2201-7},
}

@incollection {FS19,
    AUTHOR = {Fillastre, Fran\cois and Seppi, Andrea},
     TITLE = {Spherical, hyperbolic, and other projective geometries:
              convexity, duality, transitions},
 BOOKTITLE = {Eighteen essays in non-{E}uclidean geometry},
    SERIES = {IRMA Lect. Math. Theor. Phys.},
    VOLUME = {29},
     PAGES = {321--409},
 PUBLISHER = {Eur. Math. Soc., Z\"urich},
      YEAR = {2019},
      ISBN = {978-3-03719-196-5},
   MRCLASS = {51M09 (53A35)},
  MRNUMBER = {3965127},
}

@article {G20,
    AUTHOR = {Graf, Melanie},
     TITLE = {Singularity theorems for {$C^1$}-{L}orentzian metrics},
   JOURNAL = {Comm. Math. Phys.},
  FJOURNAL = {Communications in Mathematical Physics},
    VOLUME = {378},
      YEAR = {2020},
    NUMBER = {2},
     PAGES = {1417--1450},
      ISSN = {0010-3616,1432-0916},
   MRCLASS = {53C50 (53C20 83C75)},
  MRNUMBER = {4134950},
MRREVIEWER = {Clemens\ Saemann},
       DOI = {10.1007/s00220-020-03808-y},
       URL = {https://doi.org/10.1007/s00220-020-03808-y},
}

@article {KOSS22,
    AUTHOR = {Kunzinger, Michael and Ohanyan, Argam and Schinnerl, Benedict
              and Steinbauer, Roland},
     TITLE = {The {H}awking-{P}enrose singularity theorem for
              {$C^1$}-{L}orentzian metrics},
   JOURNAL = {Comm. Math. Phys.},
  FJOURNAL = {Communications in Mathematical Physics},
    VOLUME = {391},
      YEAR = {2022},
    NUMBER = {3},
     PAGES = {1143--1179},
      ISSN = {0010-3616,1432-0916},
   MRCLASS = {83C75 (53B30)},
  MRNUMBER = {4405569},
MRREVIEWER = {Stefan\ Suhr},
       DOI = {10.1007/s00220-022-04335-8},
       URL = {https://doi.org/10.1007/s00220-022-04335-8},
}

@article {KS18,
    AUTHOR = {Kunzinger, Michael and S\"{a}mann, Clemens},
     TITLE = {Lorentzian length spaces},
   JOURNAL = {Ann. Global Anal. Geom.},
  FJOURNAL = {Annals of Global Analysis and Geometry},
    VOLUME = {54},
      YEAR = {2018},
    NUMBER = {3},
     PAGES = {399--447},
      ISSN = {0232-704X},
   MRCLASS = {53C23 (53B30 53C50 53C80)},
  MRNUMBER = {3867652},
MRREVIEWER = {Benjam\'{\i}n Olea},
       DOI = {10.1007/s10455-018-9633-1},
       URL = {https://doi.org/10.1007/s10455-018-9633-1},
}

@article {KSS14,
    AUTHOR = {Kunzinger, Michael and Steinbauer, Roland and Stojkovi\'c,
              Milena},
     TITLE = {The exponential map of a {$C^{1,1}$}-metric},
   JOURNAL = {Differential Geom. Appl.},
  FJOURNAL = {Differential Geometry and its Applications},
    VOLUME = {34},
      YEAR = {2014},
     PAGES = {14--24},
      ISSN = {0926-2245,1872-6984},
   MRCLASS = {53B20 (53B30)},
  MRNUMBER = {3209534},
MRREVIEWER = {Young\ Ho\ Kim},
       DOI = {10.1016/j.difgeo.2014.03.005},
       URL = {https://doi.org/10.1016/j.difgeo.2014.03.005},
}

@article {L25,
    AUTHOR = {Ling, Eric},
     TITLE = {A lower semicontinuous time separation function for {$C^0$}
              spacetimes},
   JOURNAL = {Ann. Henri Poincar\'e},
  FJOURNAL = {Annales Henri Poincar\'e. A Journal of Theoretical and
              Mathematical Physics},
    VOLUME = {26},
      YEAR = {2025},
    NUMBER = {7},
     PAGES = {2293--2313},
      ISSN = {1424-0637,1424-0661},
   MRCLASS = {83C75 (53C50 83C05)},
  MRNUMBER = {4924413},
       DOI = {10.1007/s00023-024-01490-7},
       URL = {https://doi.org/10.1007/s00023-024-01490-7},
}

@book {M75,
    AUTHOR = {Munkres, James R.},
     TITLE = {Topology: a first course},
 PUBLISHER = {Prentice-Hall, Inc., Englewood Cliffs, NJ},
      YEAR = {1975},
     PAGES = {xvi+413},
   MRCLASS = {54-01},
  MRNUMBER = {464128},
}

@book {O83,
    AUTHOR = {O'Neill, Barrett},
     TITLE = {Semi-{R}iemannian geometry},
    SERIES = {Pure and Applied Mathematics},
    VOLUME = {103},
      NOTE = {With applications to relativity},
 PUBLISHER = {Academic Press, Inc. [Harcourt Brace Jovanovich, Publishers],
              New York},
      YEAR = {1983},
     PAGES = {xiii+468},
      ISBN = {0-12-526740-1},
   MRCLASS = {53-01 (53B30 53C50 83-02)},
  MRNUMBER = {719023},
MRREVIEWER = {N.\ V.\ Mitskevich},
}

@article {S16,
    AUTHOR = {S\"amann, Clemens},
     TITLE = {Global hyperbolicity for spacetimes with continuous metrics},
   JOURNAL = {Ann. Henri Poincar\'e},
  FJOURNAL = {Annales Henri Poincar\'e. A Journal of Theoretical and
              Mathematical Physics},
    VOLUME = {17},
      YEAR = {2016},
    NUMBER = {6},
     PAGES = {1429--1455},
      ISSN = {1424-0637,1424-0661},
   MRCLASS = {83C05 (53C50)},
  MRNUMBER = {3500220},
MRREVIEWER = {A.\ Burtscher},
       DOI = {10.1007/s00023-015-0425-x},
       URL = {https://doi.org/10.1007/s00023-015-0425-x},
}

@article {SS21,
    AUTHOR = {Schinnerl, Benedict and Steinbauer, Roland},
     TITLE = {A note on the {G}annon-{L}ee theorem},
   JOURNAL = {Lett. Math. Phys.},
  FJOURNAL = {Letters in Mathematical Physics},
    VOLUME = {111},
      YEAR = {2021},
    NUMBER = {6},
     PAGES = {Paper No. 142, 17},
      ISSN = {0377-9017,1573-0530},
   MRCLASS = {83C75 (46F99 53C50)},
  MRNUMBER = {4342989},
MRREVIEWER = {Clemens\ Saemann},
       DOI = {10.1007/s11005-021-01481-3},
       URL = {https://doi.org/10.1007/s11005-021-01481-3},
}

@article {SV16,
    AUTHOR = {Sormani, Christina and Vega, Carlos},
     TITLE = {Null distance on a spacetime},
   JOURNAL = {Classical Quantum Gravity},
  FJOURNAL = {Classical and Quantum Gravity},
    VOLUME = {33},
      YEAR = {2016},
    NUMBER = {8},
     PAGES = {085001, 29},
      ISSN = {0264-9381,1361-6382},
   MRCLASS = {53C50 (83F05)},
  MRNUMBER = {3476515},
MRREVIEWER = {Wolfgang\ Hasse},
       DOI = {10.1088/0264-9381/33/7/085001},
       URL = {https://doi.org/10.1088/0264-9381/33/7/085001},
}
\end{document}